\documentclass[12pt]{article}
\usepackage[leqno]{amsmath}
\usepackage{amssymb}
\usepackage{amsthm}
\usepackage{amsfonts}
\usepackage{graphicx}
\usepackage{enumerate}
\usepackage{mathrsfs}
\usepackage{hyperref}

\hypersetup{colorlinks=true, linkcolor=blue, urlcolor=blue,
linktocpage=true, implicit=false, pdfstartview=FitH,
pdfpagemode=None}

\newtheoremstyle{nonum}{}{}{\itshape}{}{\bfseries}{.}{ }{\thmnote{#3}}

\newtheorem{thm}{Theorem}[section]%[subsection]

\newtheorem*{lem*}{Lemma}
\newtheorem*{prop*}{Proposition}
\newtheorem*{thm*}{Theorem}
\newtheorem{cor}[thm]{Corollary}
\newtheorem{lem}[thm]{Lemma}
\newtheorem{prop}[thm]{Proposition}
\newtheorem{rem}[thm]{Remark}

\theoremstyle{definition}

\theoremstyle{nonum}

\newcommand{\R}{\mathbb R}
\newcommand{\C}{\mathbb C}
\newcommand{\CC}{\mathbb C}
\newcommand{\RR}{\mathbb C}

\newcommand{\FF}{\mathcal F}
\newcommand{\ZZ}{\mathbb Z}
\newcommand{\T}{T}
\newcommand{\B}{\mathcal B}
\newcommand{\F}{\mathbb F}
\newcommand{\supp}{\mbox{supp}}
\newcommand{\Supp}{\mbox{supp}}

\def\S{{\cal S}}

\def\alp{{\alpha}}

\def\RR{{\mathbb R}}
\def\pt{\partial}
\def\11{\mathds{1}}

\begin{document}

\title{On Multiplicative Maps of Continuous and Smooth Functions}
\author{Shiri Artstein-Avidan, Dmitry Faifman and Vitali Milman}
\date{}
\maketitle

\begin{abstract}
In this note, we study the general form of a multiplicative bijection on several families of functions defined on manifolds, both real or complex valued. In the real case, we prove that it is essentially defined by a composition with a diffeomorphism of the underlying manifold (with a bit more freedom in families of continuous functions). Our results in the real case are mostly simple extensions of known theorems. We then show that in the complex case, the only additional freedom allowed is complex conjugation.  Finally, we apply those results to characterize the Fourier transform between certain function spaces. 
\end{abstract}

\section{Introduction and main results}

The following is the simplest form of a lemma regarding
multiplicative maps. It is standard, and was used recently for
example in the paper \cite{AKM} where a characterization of the
derivative transform as an essentially unique bijection (up to
constant) from $C^1(\R)$ to $C(\R)$ which satisfies the chain rule
was derived.

\begin{lem}\label{lem-easy}
Assume that $K:\R \to \R$ is measurable, not identically zero and
satisfies for all $u,v\in \R$ that $K(uv)=K(u) K(v)$. Then there
exists some $p>0$ such that
\[K(u)=|u|^p\quad\mbox{ or }\quad K(u)=|u|^p sgn(u).\]
\end{lem}

When instead of $\R$ we have a more complicated set with a
multiplication operation, such as a class of functions, things
become more involved. This already became apparent in the papers
\cite{AAM} and \cite{AAFM}, where characterizations of the Fourier
transform were proved as a unique bijection between corresponding
classes of functions which maps products to convolutions.
Let us recall a result from the paper \cite{AAFM}. Here $\S = \S_\C(n)$
denotes the Schwartz space of infinitely smooth {\itshape rapidly
decreasing} functions $f\colon\RR^n\to\CC$, namely functions such
that for any $l\in \ZZ_+$ and any multi-index
$\alp=(\alp_1,\dots,\alp_n)$ of non-negative integers one has
$$\sup_{x\in \RR^n} \big|\frac{\pt^\alp
f(x)}{\pt x^\alp}(1+|x|^l)\big|<\infty$$ where as usual
$\frac{\pt^\alp f(x)}{\pt x^\alp}:=\frac{\pt^{|\alp|}f}{\pt
x_1^{\alp_1}\dots \pt x_n^{\alp_n}}$, $|\alp|:=\sum_{i=1}^n \alp_i$.

Let $\S'_\C(n)$ be the topological dual of $\S_\C(n)$.

\begin{thm}\label{thm:AAFM}[Alesker-Artstein-Faifman-Milman]
Assume we are given a bijective map $\T\colon \S_\C(n)\to \S_\C(n)$ which
admits an extension $\T '\colon \S'_\C(n)\to\S'_\C(n)$ and such that for
every $f \in \S_\C(n)$ and $g\in \S'_]\C(n)$ we have $\T_u (f \cdot g) = (\T
f)\cdot (\T_u g)$. Then there exists a $C^\infty$-diffeomorphism
$u\colon \RR^n\to\RR^n$ such that
\begin{eqnarray*}
\mbox{either } \T(f)=f\circ u \mbox{ for all } f\in \S_\C(n),\\
\mbox{or } \T (f)=\overline{f\circ u} \mbox{ for all }f\in \S_\C(n).
\end{eqnarray*}
\end{thm}
Thus, multiplicativity is valid only for transforms which are
essentially a ``change of variables''. One of the elements in the
proof was a lemma similar to those appearing in Appendix A below.

An obvious corollary of Theorem \ref{thm:AAFM}, which appeared in \cite{AAFM},
was a theorem characterizing Fourier transform
which
is denoted by $\FF$ and defined by
\[ (\FF f)(t) = \int_{\R} f(x) e^{-2\pi i  xt}d x.\]
It is
well known that Fourier transform exchanges pointwise product
on $\C$ with usual convolution, which is denoted by $f*g$; that is,
$\FF(f\cdot g) = \FF f * \FF g$ and vice versa.
The corollary of Theorem \ref{thm:AAFM} is that the Fourier transform is, up to
conjugation and up to a diffeomorphism,  the only one
 which maps product
to convolution among bijections  ${\cal F}
\colon \S \to \S$ which have an extension  ${\cal
F}' \colon \S'\to \S'$. It is not hard to check that if convolution
is also mapped back to product then the diffeomorphism $u$ above
must be the identity mapping, for details see \cite{AAFM}.

One of the main theorems in the present note is that
the assumption of the existence of ${\cal
F}' \colon \S'\to \S'$ may be omitted in this theorem (and the corresponding extension of ${\cal T}$ in Theorem \ref{thm:AAFM}). This is presented in Theorem
\ref{thm-main:ComplexC^k}, one instance of which is ${\cal B}$ being Schwartz space.
A direct Corollary of the theorem is

\begin{thm}\label{thm-main:schwartzfourier}
Let $T: \S_{\C}(n) \to \S_{\C}(n)$ be a bijection. \\
1. Assume $T$
satisfies
\[ T(f*g) = T f\cdot T g.\]
Then there exists a  $C^\infty$-diffeomorphism $u:\R^n \to \R^n$
such that either $T  f(u(x)) = \FF f(x)$ or $T f(u(x))=\FF
\overline{f(x)}$.\\
2. Assume $T$ satisfies
\[ T(f\cdot g) = T f* T g.\]
Then there exists a  $C^\infty$-diffeomorphism $u:\R^n \to \R^n$ such
that either $T  f = \FF (f\circ u)$ or $T f=\FF
(\overline{f\circ u})$.\\
\end{thm}
%
%\[ T(\F f * \F g) = T\F f \cdot T\F g.\]
%\[ T\F(f\cdot g) = T\F f \cdot T \F g.\]
%$ T \F  $ is multiplicative
%$ T \F f  = f\circ u $
%$ T g = \F^{-1}g \circ u$
%
%\[ T(f\cdot g) = T f* T g.\]
%\[ \F T(f\cdot g) = \F T f \cdot \F T g.\]
%$ \F T$ is multiplicative
%$ \F T f = f\circ u$
%$ Tf = \F^{-1} (f\circ u)$.
%$ Tf = \F^{-1} (\overline{f\circ u)})$

\begin{rem}
Similarly, Theorem \ref{thm-main:ComplexC^k} may also be used to characterize bijections
$T: \S_{\C}(n) \to \S_{\C}(n)$ which satisfy $T(f*g) = T f*Tg$.
\end{rem}

Let us quote one more application of the method to Fourier theory.
We denote by $C_c^\infty(\R,\C)$ the smooth complex valued function
on $\R$ which have compact support. It is well known, and referred
to as a Paley-Wiener type theorem, that the class
$C_c^\infty(\R,\C)$ is the image under Fourier transform of the
class $PW(\R)$ consisting of functions $F$ which decay on the real
axis faster than any power of $|x|$, and have an analytic
continuation on the complex plane satisfying the estimate  $|F(z)| <
A \exp (B|z|)$ for some constants $A,B$, see for example \cite{GS}. A
similar characterization holds for functions of several variables,
and we denote this class $PW(\R^n) = \FF(C_c^{\infty}(\R^n, \C))$.
%
%are restrictions to $\R$ of entire
%functions, see \cite{SW}. Let us denote this class be $E(\R, \C)$.
The following will be an immediate corollary of Theorem
\ref{thm-main:ComplexC^k}.

\begin{thm}\label{thm-main:compactfourier}
Let $T:C_c^\infty(\R^n,\C) \to PW(\R^n)$ be a bijection which
satisfies
\[ T(f*g) =T f\cdot T g.\]
Then there exists a  $C^\infty$-diffeomorphism $u:\R^n \to \R^n$
such that either $T  f(u(x)) = \FF f(x)$ or $T f(u(x))=\FF
\overline{f(x)}$.
\end{thm}

The setting of Schwartz space, and of its dual, in previous results, was very
specific, and from the point of view of merely multiplicative
mappings - not very natural. It was discussed mainly for its
application to Fourier transform. However, in other characterization
problems we found that similar tools were used in their proofs, and
it turned out that in most of the natural situations in which we
encounter multiplicative transforms, it is possible to characterize
their form. One example was already given above in the form of compactly supported infinitely smooth functions.
Below are several other such examples, and these are the main
theorems to be proven in this note. The exposition is intended to
make these tools available to the reader, more than to demonstrate
the specific results, most of which we later discovered have already
been proved in the literature (some more than 60 years ago, and some
very recently). We view our method as very straightforward and
natural, and believe it can be applied in many different situations.

%In this note (as opposed to the case of general Fourier theory),
%unless otherwise stated, functions are assumed to be real valued.

In the following, a map $T:\mathcal B\to\mathcal B$ between some
class $\mathcal B$ of real- or complex-valued functions, is called
multiplicative if $T(fg)=Tf\cdot Tg$ pointwise for all $f,
g\in\mathcal B$. Throughout the paper, we will address several
families of $C^k$ functions, defined on a $C^k$ manifold $M$. When
discussing Schwartz functions, it should always be understood that
$M=\R^n$, and $k=\infty$.

Out first theorem regards multiplicative maps on continuous real
valued functions, and it goes back to Milgram \cite{Milmgram}. We
also extend it to the class of continuous compactly supported
functions.

\begin{thm}\label{thm-main:Cmfld}
Let $M$ be a real topological manifold, and $\B$ is either $C(M,\R)$
or $C_c(M,\R)$. Let $T:\B \to \B$ be a multiplicative bijection.
Then there exists some continuous $p:M\to \RR_+$, and a
homeomorphism $u: M \to M$ such that
\begin{equation}
(Tf)(u(x)) = |f(x)|^{p(x)} sgn(f(x)).
\end{equation}
\end{thm}

\begin{rem} Without some non-degeneracy (and above we assume bijectivity,
which is very strong non-degeneracy) there is a simple
counterexample: let $Tf = f$ on $x \le 0$, let $Tf (x) = f(x-1)$ on
$x \ge 1$ and let $Tf = f(0)$ on $[0,1]$. However, this
counterexample may actually hint that in a more general situation
the map $u$ may be a set valued map.
\end{rem}

Next we move to the classes of $C^k$ functions, where similar
theorems hold, and moreover, no extra power is allowed, so that the
mapping is automatically linear. This theorem is also known, but
much more recent - it appears in \cite{semrl-mrc} for $k<\infty$.
The $C^{\infty}$ case remained open in \cite{semrl-mrc}, and our
method is able to clarify it as well. However, it also was already
settled (by a considerably different method altogether) in
\cite{sancheztimes2}. We also obtain the same results for some
subspaces of $C^k$, namely the compactly supported functions
$C_c^k$, and the Schwartz functions $\S(n)$.

\begin{thm}\label{thm-main:Ckmfld}
Let $M$ be a $C^k$ real manifold, $1\leq k\leq \infty$, and $\B$ is
one of the following function spaces: $C^k(M,\R)$, $C^k_c(M,\R)$ or
$\S_\R(n)$. Let $T:\B \to \B$ be a multiplicative bijection. Then
there exists some $C^k$-diffeomorphism $u: M \to M$ such that
\begin{equation}
(Tf)(u(x)) = f(x),
\end{equation}
In particular, $T$ is linear.
\end{thm}

We also address the case of complex-valued functions, which seems
not to have been treated in previous works.
\begin{thm}\label{thm-main:Complex}
Let $M$ be a topological real manifold, and $\B$ is either $C(M,\C)$
or $C_c(M,\C)$. Let $T:\B \to \B$ be a multiplicative bijection.
Then there exists some homeomorphism $u: M \to M$ and a function
$p\in C(M,\C)$, $Re(p)>0$ such that either
\[T(re^{i\theta})(u(x)) = |r(x)|^{p(x)} e^{i\theta(x)}\] or \[T(re^{i\theta})(u(x)) = |r(x)|^{p(x)} e^{-i\theta(x)}\]
\end{thm}

\begin{thm}\label{thm-main:ComplexC^k}
Let $M$ be a $C^k$ real manifold, $1\leq k \leq \infty$,  and $\B$
is one of the following function spaces: $C^k(M,\C)$, $C^k_c(M,\C)$
or $\S_\C(n)$. Let $T:\B \to \B$ be a multiplicative bijection. Then
there exists some $C^k$-diffeomorphism $u: M \to M$ such that either
$Tf(u(x)) = f(x)$ or $Tf(u(x))=\overline{f(x)}$. In particular, $T$
is $\RR$-linear.
\end{thm}

We give other variants of these theorems, and applications to
Fourier transform, in Section \ref{sec:applicationtofourier}.

%This note is organized as follows:

%In the next three sections we
%prove the main theorems. In Section \ref{sec:zerosets} we discuss
%how the transforms change the zero-set of a function, and show that
%this is governed by a point map. In Section \ref{sec:therealcase} we
%prove the theorems concerning real valued functions. In
%Section \ref{sec:thecomplexcase} we address the complex-valued case. In Section
%\ref{sec:applicationtofourier} we consider other subclasses of
%functions, which are relevant to Fourier transform and other
%application, state the relevant theorems and remark upon their
%proofs. In Appendix A we state and prove some functional
%Lemmas which lie in the core of the proof of the main theorems. Appendix
%B is connected to the theme of the paper but is not in  direct
%relation to any of the above theorems. In it we consider maps which
%preserve intersection of open sets in $\R^n$, and show that these
%must be induced by point maps on $\R^n$.

\subsection*{Acknowledgements}
The authors would like to thank Mikhail Sodin for several useful discussions, Bo'az Klartag for explaining to us the failure of our original method of zero sets in the $C^\infty$ $n$-dimensional setting, and the referee for numerous useful remarks and  references, which allowed us to improve our results, and helped to provide better structure and context for the article.

\section{Zero sets}\label{sec:zerosets}

In the following section, $0\leq k\leq\infty$, and $M$ is a $C^k$
real manifold. We use for $f\in C^k(M)$ the notation $Z(f)$ for the
zero-set of the function, namely $Z(f)=\{x\in M:f(x)=0\}$. We will
also need (though in a very mild manner) the notion of the ``jet''
of a function at a point; in fact, we will only need here a function
$\rho$ whose $k$-jet at a point $x_0$, denoted $J\rho (x_0)$, is
vanishing. This roughly means that all its derivatives at the point
vanish. For the precise definition of a jet, see Appendix A.
Finally, we fix a field $\F$ which is either $\R$ or $\C$. For the
remainder of the section, all functions will have values in $\F$,
and it will be often omitted from the notation.

The goal of this section is to establish the following
\begin{prop}\label{prop-compact}
Let $0\le k \le \infty$ be an integer, and let $M$ be a $C^k$
manifold. Let $\B$ be one of the following function families:
$C^k(M,\F)$, $C^k_c(M,\F)$, $\S_{\F}(n)$. Assume $T:\B\rightarrow \B$ is a
multiplicative bijection. Then there exists a homeomorphism $u:M\to
M$ such that $Z(Tf)=u(Z(f))$ for all $f\in \B$.
\end{prop}

We present two different proofs. The first only applies to
$\B=C^k(M)$ with $k<\infty$, and also in several 1-dimensional cases
for the other function families, which will be specified later. The
second proof is due to Mrcun \cite{Mrcun}, which applies in all
cases, with slight modifications for the cases  $\B=C^k_c(M,\F)$ and
$\B=\S_{\F}(n)$.

\subsection{The case of $\B=C^k(M,\F)$, $k<\infty$}
\begin{lem}\label{lem:zeros_divisible}
Let $f\in C^k(M)$. If $f(x_0)=0$, then there exists $h\in C^k(M)$
s.t. $Z(h)=\{x_0\}$, and $f^{4k+4}$ is divisible by $h$ in $C^k(M)$.
\end{lem}
\begin{proof} Fix  $\rho\in C^k(M)$ which is non-negative,
with $J\rho(x_0)=0$, and $\rho(x)>0$ for $x\neq x_0$. Take
$h=|f|^2+\rho$. It is then easy to see (Say, by induction) that
$f^{4k+4}/h\in C^k(M)$, and we are done. \end{proof}

\begin{rem}\label{rem:RandS1-I} If $M=\R$ or $M=S^1$, this also holds for
$k=\infty$ (with $f$ instead of $f^{4k+4}$): simply take $h(x)=x$
for $M=\R$, $x_0=0$. Note that the statement is local, so it applies
to $S^1$ as well. If in addition all functions are required to
belong to $\S(1)$, one can construct $h\in\S(1)$ with the required
property. Since those constructions only apply in the 1-dimensional
case, while the corresponding $C^\infty$ results hold in all
dimensions and will be proven differently, we omit the details.
\end{rem}

\begin{cor}\label{lem:dima}
Let $f\in C^k(M)$ and $x,y\in M$, $x\neq y$ s.t. $f(x)=f(y)=0$. One
can then represent $f^{4k+4}=f_1f_2f_3$ with $f_j\in C^k(M)$ s.t.
$f_1(x)=f_3(y)=0$, and $Z(f_1)\cap Z(f_3)=\emptyset$.
\end{cor}

\begin{proof}  Take, using Lemma \ref{lem:zeros_divisible} $f_1$ and $f_3$ such that
 $Z(f_1)=\{x\}$, $Z(f_3)=\{y\}$, and $f^{4k+4}$ is divisible by both $f_1$ and
$f_3$  in $C^k(M)$. It is then obvious that $(f_1f_3)$ divides
$f^{4k+4}$  in $C^k(M)$, since for all $z\in M$ either $f_1(z)\neq
0$ or $f_3(z)\neq 0$. Thus, $f_2=\frac{f^{4k+4}}{f_1f_3}\in C^k(M)$
is well-defined.
\end{proof}

\begin{lem}\label{lem:gcd}
Let $f,g\in C^k(M)$ s.t. $f(x)=g(x)=0$. Then one can find $h\in
C^k(M)$ with $h(x)=0$ s.t. both $f^{4k+4}$ and $g^{4k+4}$ are
divisible by $h$.
\end{lem}
\begin{proof}  Take $h=f^2+g^2$, and verify divisibility by induction. \end{proof}

We denote by $gcd(f^{4k+4},g^{4k+4})$ the family of all such
functions $h$.

\begin{rem}\label{rem:RandS1-II}  Again, if $M=\R$ or $M=S^1$, this also holds for
$k=\infty$, with $f, g$ instead of $f^{4k+4}, g^{4k+4}$, since if
the zero is assumed to be, say, at the point $0$, then  $h(x)=x$ is
a common divisor. \end{rem}

We next prove that there is a function $u:M\to M$ which governs the
behavior of zero-sets of functions under the transform $T:C^k(M)\to
C^k(M)$.

\begin{prop}\label{prop-dima}
Let $0\le k< \infty$ and let $M$ be a $C^k$ manifold. Assume
$T:C^k(M)\rightarrow C^k(M)$ is a multiplicative bijection. Then
there exists a homeomorphism $u:M\to M$ such that $Z(Tf)=u(Z(f))$
for all $f\in C^k(M)$.
\end{prop}
\begin{proof} \noindent{\em Step 1.} For $f\in C^k(M)$,
$Z(f)=\emptyset$ if and only if $Z(Tf)=\emptyset$.
\\Simply note that if
$Z(f)=\emptyset$ then $g=1/f\in C^k(M)$. Thus $(Tf)(Tg)=T(fg)=T(1)$,
and obviously $T(1) = 1$, so that $Z(Tf)=\emptyset$. For the reverse
implication, consider $T^{-1}$, which is multiplicative as well.

\noindent {\em Step 2.} For $f,g\in C^k(M)$ we have that   $Z(f)\cap
Z(g)\neq\emptyset$, if and only if  $Z(Tf)\cap Z(Tg)\neq\emptyset$.
\\ Take by Lemma \ref{lem:gcd} $h\in gcd(f^{4k+4},g^{4k+4})$. Denote $f^{4k+4}=vh$ and
$g^{4k+4}=wh$. Therefore $Tf^{4k+4}=TvTh$ and $Tg^{4k+4}=TwTh$. By
assumption, $Z(h)\neq\emptyset$, and therefore $\emptyset\neq
Z(Th)\subset Z(Tf)\cap Z(Tg)$.
\\For the reverse implication, consider $T^{-1}$.

 \noindent {\em Step 3.} There exists an invertible map
$u:M\rightarrow M$ such that $Z(f)=\{x\}$ implies  $Z(Tf)=\{u(x)\}$.
\\Assume $Z(f)=\{x\}$. By (1), $Z(Tf)\neq\emptyset$. If $y,z\in Z(Tf)$
and $y\neq z$, apply the previous lemma: write $Tf^{4k+4}=g_1g_2g_3$
with $g_1(x)=g_3(z)=0$, $Z(g_1)\cap Z(g_3)=\emptyset$. By
bijectivity of $T$, $g_j=T(f_j)$. Thus $f^{4k+4}=f_1f_2f_3$. By step
2, $f_1$ and $f_3$ have no common zeros, while by step 1 both $f_1$
and $f_3$ have zeros. Thus $f$ has at least two zeros, a
contradiction. We conclude that $Tf$ has a unique zero. If
$Z(f)=Z(g)=\{x\}$,
 by  step 2 $Tf$ and $Tg$ have a common (and by above unique) zero, thus $Z(Tf)=Z(Tg)$,
 i.e. $Z(Tf)=\{u(x)\}$ for some $u:M\rightarrow M$. Note that $u$
 must be invertible by observing $T^{-1}$.

 \noindent {\em Step 4.}  For $f\in C^k(M)$, $x\in M$, one has
$f(x)=0$ if and only if $Tf(u(x))=0$, i.e. $Z(Tf)=u(Z(f))$.
\\Take $f$ such that $f(x)=0$, and using Lemma \ref{lem:zeros_divisible} take $h$ dividing $f^{4k+4}$ with $Z(h)=\{x\}$.
Denote  $f^{4k+4}=hv$. Then $Tf^{4k+4}=ThTv$, and since $Th(u(x))=0$
by step 3, one has $Tf(u(x))=0$. For the reverse implication,
consider $T^{-1}$.

 \noindent {\em Step 5.} The map $u:M\to M$ is a homeomorphism.
\\For a chart $\R^n\simeq U\subset M$, take a $C^k$ function
$f$ with $Z(f)=M\setminus U$. Then $u(B)=M\setminus Z(T(f))$ is an
open set. Similarly, the preimage of a chart is open. Since the
charts form a basis of the topology, images and preimages by $u$ of
open sets are open. Therefore, $u$ is a homeomorphism. \end{proof}

\begin{rem} The construction of a homeomorphism with the
property as in Lemma \ref{prop-dima} does not extend to the general
$C^\infty(M)$ case. Nevertheless, it can be carried out when $M=\R$
or $M=S^1$ by Remark \ref{rem:RandS1-I} and Remark
\ref{rem:RandS1-II}.
\end{rem}

\subsection{The cases of $\B=C^k(M,\F)$, $C^k_c(M,\F)$ and $\S_{\F}(n)$, $0\leq k\leq\infty$}
We use the construction of $u$ from \cite{Mrcun} that is used in
\cite{semrl-mrc}. Some attention should be paid when repeating it
for the different families of functions, and we do it in full detail
for the convenience of the reader.

\begin{prop}\label{mrcun} Let $\B$ be a multiplicatively closed family of
functions s.t. $C^k_c(M)\subset \B\subset C^k(M)$. Let $T:\B\to\B$
be a multiplicative bijection, which restricts to a bijection of
$C^k_c(M)$. Then there exists a homeomorphism $u:M\to M$ s.t.
$u(Z(f))=Z(Tf)$ for all $f\in\B$.
\end{prop}

\noindent{\bf Proof.} \noindent {\em Step 1}. Recall the notion of a
characteristic sequence of functions $f_j$ at a point $x\in M$: this
a sequence $f_j\in C^k_c(M)$ s.t. $f_jf_{j+1}=f_{j+1}$, and $\bigcap
\supp(f_j)=\{x\}$. Fix some $x\in M$ and a characteristic sequence
of functions for it, $f_j$. By our assumptions, $g_j=T(f_j)$ satisfy
$g_jg_{j+1}=g_{j+1}$, so $\supp(g_{j+1})\subset \supp(g_j)$, and
those are also compact sets, so $K=\bigcap \supp(g_j)\neq\emptyset$.
We want to show that this intersection is in fact a single point.
Fix $y\in K$, a neighborhood $U$ of $y$, and choose a characteristic
sequence $\beta_j$ at $y$ with $\supp(\beta_1)\subset U$. Take
$\alpha_j=T^{-1}\beta_j$. Then $\gamma_j=f_j\alpha_j$ has compact
support and satisfies $\gamma_{j+1}=\gamma_j\gamma_{j+1}$, so
$\bigcap \supp(\gamma_j)\neq \emptyset$, but $\supp(\gamma_j)\subset
\supp(f_j)$, so $\bigcap \supp(\gamma_j)=\{x\}$ and $\gamma_j$ is a
characteristic family at $x$. In particular, $\gamma_1 f_j=f_j$ for
large $j$, and applying $T$, $\beta_1g_1g_j=g_j$. So
$\supp(g_j)\subset \supp(\beta_1)\subset U$ for large $j$. This
holds for every $U$, implying $\bigcap \supp(g_j)=\{y\}$. We claim
that $y$ depends only on $x$ and not on the choice of characteristic
sequence $f_j$. Assuming $\tilde f_j$ is another such sequence,
$f_j\tilde f_j$ is also a characteristic sequence at $x$, so if
$\bigcap \supp(Tf_j)=\{y\}$, $\bigcap \supp(T\tilde f_j)=\{z\}$ and
$y\neq z$ then $\bigcap \supp(Tf_jT\tilde f_j)\subset\{y\}\cap
\{z\}=\emptyset$, a contradiction. We thus define the map $u:M\to M$
by $u(x)=y$. By bijectivity of $T$, it is obvious that $u$ is also
bijective.

\noindent {\em Step 2}. We next claim that for all $f\in \B$, $x\in
M$, $T(f)(u(x))$ depends only on the germ $f_x$. Indeed, assume
$f_x=g_x$. Take a characteristic sequence $\phi_j$ at $x$. Then for
some large $j$, $\phi_j f=\phi_j g$, so
$T(\phi_j)T(f)=T(\phi_j)T(g)$. By construction of $u$,
$T(\phi_j)\equiv 1$ in a neighborhood of $u(x)$, so
$T(f)(u(x))=T(g)(u(x))$.

\noindent {\em Step 3} Take $f\in\B$ s.t. $f(x)\neq 0$. Choose
$g\in\B$ s.t. $(fg)_x = 1_x$ (which can be done since
$C^k_c(M,\F)\subset\B$). Then $T(f)(u(x))T(g)(u(x))=T(1_x)(u(x))$.
Now $T(f_x)(u(x))T(1_x)(u(x))=T(f_x)(u(x))$ by multiplicativity for
any germ $f_x$, implying by surjectivity of $T$ that
$T(1_x)(u(x))\neq 0$ (in fact, since $1_x^2=1_x$, we immediately
conclude that $T(1_x)(u(x))=1$). Thus $T(f)(u(x))\neq 0$. By
considering $T^{-1}$, we get $f(x)\neq 0\iff Tf(u(x))\neq 0$, as
required. Finally, $u$ is a homeomorphism by Step 5 of the proof of
Proposition \ref{prop-dima}.$\hfill \square$

\begin{rem}
It follows from step 2 that in fact $T$ maps germs of functions at
$x$ to germs of functions at $u(x)$. This is also an immediate
consequence of Proposition \ref{prop-compact}, as will be seen in
the next section.
\end{rem}

\begin{cor}
Let $\B$ be a multiplicatively closed family of functions s.t.
$C^k_c(M)\subset \B\subset C^k(M)$ and for all $f\in\B$, $\{x\in
M:f(x)=1\}\subset M$ is compact. Let $T:\B\to\B$ be a multiplicative
bijection. Then there exists a homeomorphism $u:M\to M$ s.t.
$u(Z(f))=Z(Tf)$ for all $f\in\B$.
\end{cor}

\noindent{\bf Proof.} By the Proposition above, it only remains to
verify that $T$ restricts to a bijection of $C_c^k(M)$. Observe that
\[f\in C^\infty_c(M)\iff fg=f\mbox{ for some }g\in \B\]
Indeed, if $f\in C^\infty_c(M)$, just choose any $g\in C_c^k(M)$
with $g\equiv 1$ on $\supp(f)$. In the other direction, if $fg=f$
with $g\in\B$, then $g\neq 1$ outside some compact set $K$, implying
$\supp(f)\subset K$. Since $T$ is multiplicative and bijective,
$fg=f\iff TfTg=Tf$, so $T$ restricts to a bijection of $C^k_c(M)$ as
required. $\hfill \square$

\begin{cor} For both $\B=C^k_c(M,\F)$ and $\B=\S_{\F}(n)$,
a multiplicative bijection $T:\B\to\B$ defines a homeomorphism
$u:M\to M$ s.t. $u(Z(f))=Z(Tf)$ for all $f\in\B$.
\end{cor}
\noindent{\bf Proof.} Simply apply the Corollary above. In fact, for
$\B=C^k_c(M,\F)$ Proposition \ref{mrcun} applies immediately.
$\hfill \square$

\begin{rem} The family $\B=C^k(M,\F)$ does not satisfy all assumptions
automatically: it is not immediate that $T:C^k(M)\to C^k(M)$
preserves the subspace of compactly supported functions, which makes
the construction in \cite{Mrcun} slightly more involved. We do not
repeat here the proof in this case, since no details of the original
proof should be modified.
\end{rem}

\section{Real valued functions}\label{sec:therealcase}
In the following section, we describe the general form of $T:\B\to
\B$ for the real-valued function families $\B$ from Proposition
\ref{prop-compact}. We then separately treat the cases of $k=0$ and
$0<k\leq\infty$. The reader might want to review the notion of the
jet of a function before proceeding (see Appendix A). We will write
$\B^k$ instead of $\B$ for any of the families $C^k(M)$, $C^k_c(M)$,
and also $\S(n)$ if $k=\infty$. One always has $C^k_c(M)\subset
\B^k$. In the following, $f_x$ denotes the germ of $f$ at $x$.
\begin{prop}\label{propT_u123}
Given a multiplicative bijection $T:\B^k\to \B^k$ (with $0\leq k\leq
\infty$) there exists a homeomorphism $u$, given by Lemma
\ref{prop-compact}, such that letting $T_u$ be defined by
$T_u(f)=T(f)\circ u$, the new map $T_u:\B^k\to C(M, \R)$ has the
following properties:
\\(1) It is
multiplicative, that is,  $T_u(fg)=T_u(f)T_u(g)$
\\(2) It is local, namely $(T_uf)_x=(T_ug)_x$ when $f_x=g_x$. Moreover, $(T_uf)(x)=F(x,
J^kf(x))$.
\\(3) It is determined by its action on non-negative functions, namely
\[T_u(f)(x)=\left\{\begin{array}{ll}0,& f(x)=0\\
T_u(|f|)(x) sgn(f(x)),& f(x)\neq 0\end{array}\right.\]
\end{prop}
\begin{proof}
Part (1) is obvious. For part (2), observe:

\noindent {\em Step 1}. $T(0) = 0$. Immediate since $T(f)T(0)=T(0)$ for all $f$, and $T$ is bijective.

 \noindent {\em Step 2}. For any open set $V$, $f =
g$ on $V$ implies $Tf = Tg$ in $cl(u(V))$.
\\
Indeed, take any open ball $B\subset V$, and take a function $h$
with $Z(h)=M\setminus B$ (a bump function over $B$). We have $f\cdot
h = g\cdot h$ in $M$, and so $ T f \cdot T h = T g \cdot T h$ in $M$
and by Proposition \ref{prop-dima} $Z(Th)=M\setminus u(B)$, implying
$Tf=Tg$ in $u(B)$. This holds for all $u(B)\subset u(V)$; since $u$
is a homeomorphism, $Tf=Tg$ in $u(V)$, and by continuity in
$cl(u(V))$.

Put another way, we proved that the germ $(Tf)_{u(x)}$ only depends
on the germ $f_x$ of $f$ at $x$. We may write
$(Tf)_{u(x)}=T(f_x)_{u(x)}$, and $T(f)(u(x))=T_u(f)(x)=T_u(f_x)(x)$.
Thus we may compute $T_u(C_x)$ for the constant germ $C$ at $x$, even
if the constant function $C\notin \B$, by completing $C$ to a
compactly supported function away from $x$.

\noindent {\em Step 3}. $(Tf)(u(x))  =  F(x, J^kf(x))$ for some $F:
J^k \to \RR$.\\ Indeed, fix $x_0\in M$. Choose an open ball $U$
around $x_0$, and two open sectors $V_1,V_2\subset U$, having $x_0$
as a common vertex, and $cl(V_1)\cap cl(V_2)=\{x_0\}$. Given two
functions $f_1,f_2\in C(M)$, assume that $J^kf_1(x_0)=J^kf_2(x_0)$.
By Whitney's extension theorem, one can choose a $C^k_c(M)$ function $f_3$ that equals $f_j$ on $V_j$ for
$j=1,2$. Then $T_u(f_3)$ and $T_u(f_j)$ coincide on $cl(V_j)$,
and in particular, $Tf_1(u(x_0))=Tf_3(u(x_0))=Tf_2(u(x_0))$.
Therefore, $(T_uf)(x)=(Tf)(u(x))= F(x, J^kf(x))$. This completes the
proof of (2).

\noindent {\em Step 4}.  $T_u(-1_x)=((-1)^{\delta(x)})_x$ with
$\delta(x)\in\{0,1\}$ a locally constant function. Indeed, for any
germ $f_x$, $T_u(f_x)= T_u(f_x)T_u(1_x)$ so $T_u(1_x)\equiv 1_x$. Again
by multiplicativity,  $T_u(-1_x)^2=T_u(1_x)=1_x$, so
$T_u(-1_x)=((-1)^{\delta(x)})_x$ Finally, note that $T_u(-1_x)$ is the
germ of a continuous function to conclude $\delta(x)$ is locally
constant.

\noindent {\em Step 5}.  \[(T_uf)(x)=\left\{\begin{array}{ll}0,& f(x)=0\\
T_u(|f|)(x)T_u(sgn(f(x))),& f(x)\neq 0\end{array}\right.\] Indeed,
since by continuity, for any $x\in M$ such that $f(x)\neq 0$, $sgn
f$ is locally constant at $x$, and by step 2 $T_u(f_x)(x)=T_u(sgn f(x)
|f_x|)(x)$.

\noindent {\em Step 6}. The function $\delta(x)$ from step 4
satisfies $\delta(x) = 1$ for all $x$. Indeed, $f(x_0)>0$ implies
$(T_uf)(x_0)>0$: One can choose $g\in C_c^k(M)$ with $f(x)=g(x)^2$
for $x$ near $x_0$. Then $T_u(f)(x_0)=T_u(g)^2(x_0)>0$. If
$T_u(-1)(x)=+1$ for some $x$, it implies that
$(Tf)(u(x))=T(|f|)(u(x))$ is always non-negative on the connected
component of $u(x)$ in $M$, thus contradicting surjectivity of $T$.
Therefore, $T_u(-1_x)= -1_x$.
\end{proof}

From now on we work with $T_u$ instead of $T$, and only return to the
original $T$ when we show that $u$ is a $C^k$-diffeomorphism. Thus,
for now we cannot assume that the image of $T_u$ is $C^k$, but only
that it is continuous.

By part (3) of Proposition \ref{propT_u123}, we need to study our
transform only on non-negative functions.

\begin{lem}\label{lem:general_form}
Let $T_u$ satisfy the conclusion of Proposition \ref{propT_u123}.
Then there exists a global section of $(J^k)^*$, $c_k$, such that
for $f(x_0)>0$, \[T_u(f)(x_0)=\exp( \langle c_k(x_0) , J^k(\log
f)(x_0)\rangle)\] If
$k<\infty$, $c_k$ is a continuous global section.  If $k=\infty$,
it is locally finite dimensional and continuous, i.e.
every $x_0\in M$ has an open neighborhood $U$ such that $c_{\infty}=Q_n(c_n)$ in
$U$ for some finite $n$, and $c_n$ is a continuous section  of
$(J^n)^*$ over $U$.
\end{lem}
\begin{proof}
 As in the proof above, $T_u$ clearly maps positive functions to
positive functions. Define $A:C^k_c(M)\to C(M)$ by $A(f)= \log
T_u(\exp(f))$. Then $A$ is an additive transformation, with the
additional property that $A(f)(x)=B(x, J^kf(x))$, where $B(x,\cdot):
J_x^k\to\R$ is an additive functional for every $x\in M$. Apply
Lemma \ref{lemma6.2} from Appendix A to conclude the stated result.
\end{proof}

We are ready to conclude the proof of Theorem \ref{thm-main:Cmfld}.
\begin{proof}[Proof of Theorem \ref{thm-main:Cmfld}]
 Recall that $T_u=T\circ u$, and observe that it is
surjective, as $T$ is. We already
know by Lemma \ref{lem:general_form} with $k=0$ that \[(Tf)(u(x))=\left\{\begin{array}{ll}0,& f(x)=0\\
|f(x)|^{c_0(x)}sgn(f(x)),& f(x)\neq 0\end{array}\right.\] with
$c_0(x)$ continuous. We are left to show that $c_0(x)>0$ everywhere.
Indeed, if $c_0(x)=0$ then $T_uf(x)$ is either $0$ or $\pm 1$ for every $f$, contradicting surjectivity of $T_u$; while if $c_0(x)<0$, we could
take a positive function $f$ with an isolated zero at $x_0$, and then $\lim_{x\to x_0} Tf(x)=\infty$, contradicting continuity of $Tf$.

Next we assume $k\geq 1$ and prove Theorem
\ref{thm-main:Ckmfld}, i.e. that $(Tf)(u(x))=f(x)$. We will denote $v=u^{-1}$.
\\Fix some $x_0\in M$, and choose a relatively compact neighborhood $U$ of $x_0$ as in Lemma \ref{lem:general_form}.
Thus $(T f)(u(x))=\exp( \langle c_k(x) , J^k(\log f)(x)\rangle)$ for
positive $f\in C^k(U)$, and $c_k=Q(n,k)(c_n)$ for some finite $n$,
$c_n$ a continuous section of $J^n$ over $U$ (i.e., $c_k$ only
depends on the $n$-jet of the function). We claim that one can take
$n=0$. We may assume that $U$ is a coordinate chart with $x_0$ at
the origin. Then, if $c_n$ at $x_0$ depends on terms of the jet
other than the constant term, one has
\[ \langle c_k(x) , J^k(\log
f)(x)\rangle=a_0(x)\log f(x)+\sum_{1\leq|\alpha|\leq n}a_\alpha(x)\frac{\partial ^{|\alpha|}\log f }{\partial x^{\alpha}}\] where $a_\alpha\in C(U)$, and $a_{\alpha_0}(0)\neq 0$ for some $\alpha_0\neq 0$. Fix such $\alpha$ of maximal modulus $m=|\alpha|$, and take $f(x)=\lambda_1 x_1+...+\lambda_dx_d$ where $d=\dim M$. Then \[\frac{\partial ^{|\alpha|}\log f }{\partial x^{\alpha}}=\frac{\pm \prod \lambda_j^{\alpha_j}(|\alpha|-1)!}{(\sum \lambda_j x_j)^{|\alpha|}}\] so an appropriate choice of $\lambda_j$ (not all zero) will guarantee that \[\sum_{|\alpha|=m} a_\alpha(x)\frac{\partial ^{|\alpha|}\log f }{\partial x^{\alpha}}=\frac{C(x)}{(\sum \lambda_j x_j)^{m}}\] with $C(x)$ continuous and non-vanishing near $0$. The same would hold, with a different $C(x)$, also if we sum up all the $\alpha$-derivatives for $1\le |\alpha|\le n$. It follows that
\[T_u(\sum \lambda_j x_j)=|\sum \lambda_j x_j|^{a_0(x)}e^{\sum_{1\le |\alpha|\le n} a_\alpha(x) (\log \sum\lambda_j x_j)^{(\alpha)}}\] cannot be continuous at $0$, a contradiction.

Thus $(Tf)(x)=f(v(x))^{a_0(v(x))}$ for positive $f$. Taking
$f_x\equiv 2_x$, we  conclude that $a_0(v(x))\in C^k(M)$. As in the
case $k=0$ we see that $a_0>0$, so $f(v(x))\in C^k(M)$ for all
positive $f$, which implies $v\in C^k(M, M)$. The same reasoning
applied to $T^{-1}$, we conclude that $u$ is a $C^k$-diffeomorphism.
Now it is obvious that $C_c^k(M)$ is invariant under $T$, so
$T:C_c^k(M)\to C_c^k(M)$ is a bijection. Since $k \ge 1$, we must
have $a_0\equiv 1$, so $(Tf)(u(x))=f(x)$, as claimed.
\end{proof}

\section{Complex valued functions}\label{sec:thecomplexcase}
In this section, we describe the general form of $T:\B\to \B$ for
the complex-valued function families $\B$ from Proposition
\ref{prop-compact}. We again treat the cases of $k=0$ and
$0<k\leq\infty$ separately. We write $\B^k$ instead of $\B$ for any
of the families $C^k(M)$, $C^k_c(M)$, and also $\S(n)$ if
$k=\infty$. One always has $C^k_c(M)\subset \B^k$.

\begin{prop}\label{propT_u123complex}
Given a multiplicative bijection $T:\B^k\to \B^k$ (with $0\leq k\leq
\infty$) there exists a homeomorphism $u$, given by Lemma
\ref{prop-compact}, such that letting $T_u$ be defined by
$T_u(f)=T(f)\circ u$, the new map $T_u:\B^k\to C(M, \C)$ has the
following properties:
\\(1) It is multiplicative: $T_u(fg)=T_u(f)T_u(g)$
\\(2) It is local, namely $(T_uf)_x=(T_ug)_x$ when $f_x=g_x$. Moreover, $(T_uf)(x)=F(x,
J^kf(x))$.
\\(3) $T_u(f)(x)=0$ if and only if $f(x)=0$.
\end{prop}

The proof of this statement is as in the real case, and is omitted.
We denote $v=u^{-1}$. It is then obvious that $C^k_c(M,\C)$ is an
invariant subspace of $T$, on which $T$ is bijective. Also by part
(2), $T_u$ extends naturally to a map $T_u:C^k(M,\C)\to C(M,\C)$ which
retains properties (1)-(3).
%Moreover, $T_u$ remains injective in
%every stalk of $C^k$ functions: $T_u:\mathcal O^k_x\to \mathcal
%O^0_x$ is injective (here $\mathcal O^j$ denotes the sheaf of $C^j$
%functions on $M$). This statement is of course independent of the
%extension of $T_u$. And it holds

We next make some helpful decompositions, which enable us to treat
the various parts of the transform separately. A complex valued
function $f\in C^{k}(M, \C)$ can be written as $r(x)e^{i\theta(x)}$
where  $r \geq 0 $ continuous s.t. $r\in
C^k(\{x: r(x)>0\}$, and $\theta\in C^k(\{x:r(x)>0\}, S^1)$ and
\begin{prop}\label{propT_u123decomp} There exists a function $g_0\in C(M, \R_+)$, and global sections $d_k\in (J^k)^*(M\times S^1,\R)$,  $h_k\in (J^k)^*(M\times\R)$ and $e_k\in (J^k)^*(M\times S^1, S^1)$ such that
\[T_u(re^{i\theta})(x)=\left\{\begin{array}{ll}0,& r(x)=0\\
r(x)^{g_0(x)}e^{i\langle h_k(x), J^k\log r(x)\rangle} e^{\langle
d_k(x), J^k\theta(x)\rangle} e^{i\langle e_k(x),
J^k\theta(x)\rangle},& r(x)\neq 0\end{array}\right.\] The sections
$h_k, d_k, e_k$ are continuous when $k<\infty$, and locally finite
dimensional and continuous when $k=\infty$.
\end{prop}

\begin{proof} We may, using multiplicativity of $T_u$,
write \[T_u(r(x)\exp(i\theta(x)))=T_u(r(x))S(\theta(x))\] where
$S(\theta)=T_u(\exp(i\theta))$. Since we already know that zeros are
mapped to zeros, $T_u:C^k(M,\R_+)\to C(M,\C^*)$ and $S:C^k(M,S^1)\to
C(M,\C^*)$  are group homomorphisms. Denote further $T_u(r)=G(r)H(r)$
and $S(\theta)=D(\theta)E(\theta)$, where $D:C^k(M,S^1)\to
C(M,\R_+)$, $E:C^k(M,S^1)\to C(M, S^1)$, $G:C^k(M,\R_+)\to
C(M,\R_+)$, $H:C^k(M,\R_+)\to C(M, S^1)$ are homomorphisms of
groups. Furthermore, property (2) immediately implies that $D, E,
G,H$ are all local, namely depend only on the jets of the functions.

As in the real case, we apply lemma \ref{lemma6.2} of Appendix A to
conclude that $G(r)(x)=r(x)^{g_0(x)}$ with $g_0\in C(M, \R_+)$ (if
$k\geq 1$, $G$ cannot depend on higher derivatives of $r$ as in the
proof of Theorem \ref{thm-main:Ckmfld}, while $g_0>0$ as in the
proof of Theorem \ref{thm-main:Cmfld}: $g_0\geq 0$ to guarantee
continuity of $T_u(r(x))$ at a zero point of $r$, and $g_0(x)=0$
would immediately contradict surjectivity of $T$). Then by Lemmas
\ref{lemSR}, \ref{lemRS}, \ref{lemSS} of Appendix A,
$D(\theta)(x)=\exp(\langle d_k(x), J^k\theta(x)\rangle)$;
$H(r)(x)=\exp(i\langle h_k(x), J^k\log r(x)\rangle)$; and
$E(\theta)=\exp(i\langle e_k(x), J^k\theta(x)\rangle)$ where $d_k,
h_k, e_k$ are as stated.
\end{proof}

We now can complete the proof of Theorem \ref{thm-main:Complex}.
\begin{proof}[Proof of Theorem  \ref{thm-main:Complex}]
Recall that $k=0$. Then $d_k(x) = 0$, $e_k(x) = m$ for some fixed
$m\in\mathbb Z$ and $h_k(x) = h_0$ for some $h_0 \in C(M)$. Thus $T_u(re^{i\theta})=r^{g_0}e^{i(m\theta+h_0\log
r)}$, $g_0, h_0\in C(M)$. From injectivity of $T_u$ on $\B$, $m=\pm
1$: otherwise, either $m=0$ and $T_u$ does not depend on $\theta$; or
$|m|\geq 2$, so $\theta$ can be replaced with $\theta+2\pi/m$
without affecting $T_u(re^{i\theta})$ (for any compactly supported
$r(x)$). Thus
\[T(re^{i\theta})(u(x))=r(x)^{p(x)}e^{\pm i\theta}\] where
$p(x)=g_0(x)+ih_0(x)$, as claimed.
\end{proof}

Next we proceed to prove Theorem \ref{thm-main:ComplexC^k}.
\begin{proof}[Proof of Theorem  \ref{thm-main:ComplexC^k}]
\noindent {\em Step 1}. We prove that $u$ is a $C^k$-diffeomorphism
of $M$. Taking $r\equiv 2$, $\theta\equiv 0$ we see that
$|T(2)(x)|=2^{g_0(v(x))}\in C^k(M)$, in particular $g_0(v(x))\in
C^k(M)$. Thus for any $r(x)\in C^k(M, \R_+)$ (again $\theta(x)\equiv
0)$)
\[\log |T(r)(x)|=g_0(v(x))\log r(v(x))\in C^k(M)\] so also $r(v(x))\in C^k(M,\R_+)$. Thus $v\in C^k(M,M)$.
By considering $T^{-1}$, $u$ is also $C^k$, implying $u$ is a
$C^k$-diffeomorphism of $M$. From now on we only consider
$T_uf=Tf\circ u$, and prove that $T_uf=f$ or $T_uf=\overline f$. Note
that $T_u:C^k(M)\to C^k(M)$, and its restriction $T_u:C_c^k(M)\to
C_c^k(M)$ is a bijection.

\noindent {\em Step 2}. We show here that $d_k=0$.  If $d_k\neq 0$,
one could choose for any $x_0\in M$ a function $\theta\in
C^k(U\setminus \{x_0\})$ where $U$ is a small neighborhood of $x_0$
contained in a coordinate chart with its origin at $x_0$, for which
$d_k$ only depends on the $m$-jet, $m<\infty$, s.t. $\langle
d_k(x_n), J^k\theta(x_n)\rangle\geq \frac{g_0(x_n)}{|x_n|^2}$ for
some $U\setminus \{x_0\} \ni x_n\to x_0$, while $|J^j\theta(x)|\leq
C_j|x|^{-N_j}$ for all $j\leq k$ and $x\in U\setminus \{x_0\}$ (simply
take $\theta=C|x|^{-2}$ with appropriate $C$ in a contractible
neighborhood of the sequence $(x_n)$, and extend it to $U\setminus
x_0$ arbitrarily). Then, taking $r(x)=\exp(-1/|x|^2)$, one has
$f(x)=r(x)e^{i\theta(x)}\in C^k(U)$, $f(x_0)=0$ while
\[|T_uf(x_0)|=\lim_{n\to\infty} r(x_n)^{g_0(x_n)}\exp(\langle
d_k(x_n), J^k\theta(x_n)\rangle)\geq\]
\[\geq\lim_{n\to\infty}\exp(-g_0(x_n)/|x_n|^2)\exp(g_0(x_n)/|x_n|^2)=1\] a contradiction. Thus, $d_k\equiv 0$.

\noindent {\em Step 3}. We show that $g_0\equiv1$. Indeed, since
$|T_u(r(x)e^{i\theta(x)})|=r(x)^{g_0(x)}\in C^k(M,\R_+)$ for all
$r\in C^k(M,\R_+)$, and since $T_u$ is surjective on $C_c^k(M,\C)$,
we must have $g_0\equiv1$

\noindent {\em Step 4}. We next claim that $h_k\equiv 0$. Note that
for all $r(x)\in C^k(M,\R_+)$, \[T_u(r(x))=r(x)e^{i\langle h_k(x),
J^k\log r(x)\rangle}\in C^k(M)\] so $h_k$ is a $C^k$ section of
$(J^k)^*(M\times\R)$. First, $h_k$ only depends on the constant term
of the jet, or else $T_u(r(x))$ would not be in $C^k$ for all $r\in
C^k(M,\R_+)$: this is obvious when $k<\infty$; and if $k=\infty$, we
proceed as was done in the real case. Fix a coordinate chart $U$
s.t. $h_k(x)=(a_\alpha(x))_{|\alpha|\leq m}$ where $a_\alpha\in
C^k(U)$, $m\geq 1$ and $a_\alpha(0)\neq 0$ for some $\alpha$ with
$|\alpha|=m$. Take  $f(x)=\lambda_1 x_1+...+\lambda_dx_d$ with
coefficients $\lambda_j$ s.t.  \[\sum_{|\alpha|=m}
a_\alpha(x)\frac{\partial ^{|\alpha|}\log f }{\partial
x^{\alpha}}=\frac{C(x)}{(\sum \lambda_j x_j)^{m}}\] with $C(x)\in
C^k(U)$ and non-vanishing near $0$. Then, assuming $\lambda_1\neq 0$
and considering only points $x$ where $f(x)>0$,
\[\begin{array}{ll}
\frac{\partial }{\partial x_1}(T_uf)=\lambda_1e^{i\langle h_k(x),
J^k\log f(x)\rangle}+ \\+ (\sum\lambda_j
x_j)\left(\left(\frac{\partial C(x) }{\partial x_1}\frac{1}{(\sum
\lambda_j x_j)^{m}}-\frac{m\lambda_1 C(x)}{(\sum \lambda_j
x_j)^{m+1}}\right)+\sum_{|\alpha|\leq m} \frac{C_\alpha(x)}{(\sum
\lambda_j x_j)^{|\alpha|}}\right)e^{i\langle h_k(x), J^k\log
f(x)\rangle}\end{array}\] where all $C_\alpha$ are continuous. Thus
there is no limit to $\frac{\partial }{\partial x_1}(T_uf)$ as $x\to
0$ (along points of positivity for $f$), a contradiction. So
$h_k=h_0\in C(M,\R)$.
\\\\Then, if $h_0(x_0)\neq 0$, take a chart with $x_0$ at the origin,
and consider $f(x)=x_1\in C^k(M)$ - the first coordinate function.
Then \[\frac{\partial}{\partial x_1}(T_uf)(0)=
\lim_{x_1\to0^+}\frac{x_1e^{ih_0(x)\log
x_1}}{x_1}=\lim_{x_1\to0^+}e^{ih_0(x)\log x_1}\] which diverges
since $h_0(x)\log x_1$ is continuous when $x_1\in(0,\infty)$, and
\\$\lim_{x_1\to0^+} |h_0(x)\log x_1|=\infty$. This is a
contradiction.

\noindent {\em Step 5}. Finally, we want to show that $\langle
e_k(x), J^k\theta(x)\rangle=\pm \theta(x)$. First, by considering a
coordinate chart and polynomial functions $\theta$, we see that the
components of $e_k(x)$ are in fact $C^k$. We now treat separately
the cases $k<\infty$ and $k=\infty$.

\noindent {\em Case 1}: $k<\infty$. Then
$T_u(e^{i\theta(x)})=\exp(i\langle e_k(x), J^k\theta(x)\rangle)\in
C^k (M, S^1)$ for all $e^{i\theta}\in C^k(M, S^1)$, which is
impossible unless $e_k$ only depends on $J^0\theta$, i.e. $\langle
e_k(x), J^k\theta(x)\rangle=m\theta(x)$. From injectivity of $T_u$ on
$C_c^k(M,\C)$, $m=\pm1$.

\noindent {\em Case 2}: $k=\infty$. Let us prove the following
\begin{prop*}  $P:C^\infty(M,\R)\to
C^\infty(M,\R)$, $\theta\mapsto \langle e_k(x),
J^\infty\theta(x)\rangle$ induces an isomorphism of stalks of smooth
functions at every $p\in M$. \end{prop*}
\begin{proof} Indeed, fix $p\in M$, and a germ $u_p$ represented by $u\in
C^\infty(U,\R)$ with some small neighborhood $U\ni p$.
\\To see that $Ker(P_p)=0$, assume
$(Pu)_p=0$, so $Pu$ vanishes identically in some neighborhood $V$ of
$p$. Take any $r\in C_c^\infty(M,\R_+)$ with $\supp(r)\subset V$ and
$r(p)>0$. Then $T_u(r\exp(iu))=r\exp(iPu)\equiv r$ and also
$T_u(r)=r$, so by injectivity of $T_u$, $u$ must be a multiple of
$2\pi$ whenever $r\neq 0$, in particular $u_p\equiv (2\pi l)_p$ - a
constant germ with $l\in\mathbb Z$. Since $P(c)=m c$ for constant
functions $c$, and $m\neq 0$ from injectivity of $T_u$ on
$C_c^k(M,\C)$, we may conclude that $c=0$ and so $u_p=0$.
\\For surjectivity of $P_p$, let us find $v_p$ s.t. $(Pv)_p=u_p$. Choose any smooth continuation of $u$ to
$M$, and some $r(x)\in C_c^\infty(M,\R_+)$ with $r(p)=1$. By
surjectivity of $T_u$, one can find $v\in C^\infty(M,\R)$ s.t.
$r\exp(iu)=T_u(r\exp(iv))=r\exp(iPv)$. Thus $Pv\equiv u$ modulo
$2\pi$ in some neighborhood $V=\{r(x)\neq 0\}$ of $p$. We then may
replace $v$ by $v+2\pi l$ if necessary, and replace $V$ by a
connected neighborhood of $p$ so that $(Pv)_p\equiv u_p$, as
required.\end{proof}

\noindent {\em Step 6}. Now fix a small neighborhood $W$ in $M$,
s.t. $\langle e_k(x), J^k\theta(x)\rangle$ only depends on a finite
jet, i.e. $P$ is a differential operator in $W$. We apply a
consequence of Peetre's theorem (see Lemma \ref{peetre} below) to
conclude that $P$ is of order $0$, implying $J^k\theta=m\theta$.
From injectivity of $T_u$ on $C_c^k(M,\C)$, $m=\pm 1$. This concludes
the proof of Theorem \ref{thm-main:ComplexC^k}.\end{proof}
% We are only able to do this in two
%cases: $k<\infty$ and $M=\R$. First, by considering a coordinate
%chart and polynomial functions $\theta$, we see that the components
%of $e_k(x)$ are in fact $C^k$.
%\\If $k<\infty$, then $T_u(e^{i\theta(x)})=\exp(i\langle e_k(x), J^k\theta(x)\rangle)\in C^k (M, S^1)$ for all $e^{i\theta}\in C^k(M, S^1)$,
%which is impossible unless $e_k$ only depends on $J^0\theta$, i.e. $\langle e_k(x), J^k\theta(x)\rangle=m\theta(x)$.
%From injectivity of $T_u$, $m=\pm1$.
%\\If $k=\infty$ and $M=\R$, we immediately get a contradiction to injectivity of $T_u$ if $e_k$ depends on non-constant terms of the jet:
%We can then solve in the neighborhood of $x_0$ the ODE \[\langle e_k(x), J^k\theta(x)\rangle=0\] to obtain infinitely many different solutions.
%Take $r(x)\in C^\infty(M, \R_+)$ supported inside that neighborhood.
%For any such $\theta$, $T_u(re^{\theta})=r(x)$, a contradiction.

\section{Various generalizations and applications to Fourier transform}\label{sec:applicationtofourier}

One of our main motivations for studying multiplicative transforms
is that Fourier transform can be, in certain settings, characterized
by the property that it carries product to convolution, as explained
in the introduction. We already stated two corollaries of Theorem
\ref{thm-main:ComplexC^k} following from this point of view, namely
Theorems \ref{thm-main:compactfourier} and
\ref{thm-main:schwartzfourier}. Of course, in general the Fourier
transform is not defined on all continuous functions, and even when
it is, its image is usually not as well understood as in the case of
Schwartz and compactly supported function. Let us first state a
formal corollary of Theorem \ref{thm-main:Cmfld} and of Theorem
\ref{thm-main:Ckmfld}, in the case where $M = S^1$, that is, of
$2\pi$-periodic real-valued functions.

Denote the subclass of $\ell_2(\ZZ)$ consisting of those sequences
which are the coefficients of the fourier series of $2\pi$-periodic
real valued  $C^{k}$ functions by $E_k$. Fourier series is a
bijection $\hat{\cdot}: C^{k}(S^1, \R) \to E_k$, given by
\[ \hat {f}(n) = \int_0^{2\pi} f(x) e^{-2\pi i n x }d x.\]
It satisfies that $\widehat{f\cdot g} = \hat{f} * \hat{g}$ where here we
use $*$ for two series to mean their convolution (or Cauchy
product), that is,
\[ \{ a_n\} * \{ b_n\} = \{c_n\}\qquad {\rm where} \,\,\, c_k =
\sum_{j\in \ZZ} a_j b_{k-j}.\]

\begin{cor}\label{cor:Ccaseseries}
Let $F: C^{k}(S^1, \R) \to  E_k$ be a bijection which satisfies
\begin{eqnarray}
F(f \cdot g)=(Ff)* (Fg).
\end{eqnarray}
Then there exists some continuous $p:\RR_+\to \RR_+$ and a
$C^k$-diffeomorphism $u: S^1 \to S^1$ such that
\begin{equation}
Ff = \hat{g} \,\, {\rm with~}\, g(x) = |f(u(x))|^{p(x)}
sgn(f(u(x))).
\end{equation}
and if $k\geq 1$ then $p\equiv1$ , i.e. $g=f\circ u$.
\end{cor}

Another interesting class  to work with is
$L_2(\R) \cap C^{k}(\R)\cap L_{\infty}(\R)$, and
although our theorems do not formally apply to this class, it is not hard to check that
their corresponding variants are valid as well, as
the interested reader may care to verify.
%[Dima - Vitali - do
%we want to check this for ourselves? add a few words?]

One may thus apply the Fourier transform $\FF$ in this case, and
conclude that the only bijections from this class  to its images under
$\FF$ which map product to convolution are the standard Fourier transform
composed with the additional terms coming from out main theorems
(a diffeomorphism $u$ only, if $k\ge 1$, and some power $p(x)$ and sign if $k = 0$).
%
%
%\begin{thm}\label{thm-main:gen-Ccase}
%Let $A\subset C(\R)$ denote one of the following classes: $A_1 = L_2(\R) \cap C(\R)$
%, $A_2   =
%L_2(\R) \cap C(\R)$. Let $T:A \to  A$ be a bijection which preserves
%pointwise multiplication, that is
%\begin{eqnarray}
%T(f \cdot g)=(Tf)\cdot (Tg).
%\end{eqnarray}
%Then there exists some continuous $p:\RR_+\to \RR_+$ and a
%homeomorphism $u: \RR \to \RR$ such that
%\begin{equation}
%Tf(x) = |f(u(x))|^{p(x)} sgn(f(u(x))).
%\end{equation}
%\end{thm}
%
%\begin{rem}
Clearly  not every choice of $u$ and $p$
will give a bijection, but the statement is only on the existence of
such functions. The ``permissible'' $u$ and $p$ are to be determined by
the class in question.

We next briefly present an observation regarding possible generalizations of the main theorems. We state them in the simplest case of continuous functions on $M$.
\begin{thm}\label{thm-gen}
Let
$M$ be a real topological manifold,
let $V, W, U:C(M) \to  C(M)$ satisfy that $V$ is a bijection and
 that for all $f,g \in C(M)$
%which preserves pointwise multiplication, that is
\begin{eqnarray}\label{eq_multthreediff}
V(f \cdot g)=(Wf)\cdot (Ug).
\end{eqnarray}
Then there exist continuous
$a,b:M \to \R^+$,
$p:M\to \RR_+$, and a
homeomorphism $u: M \to M$ such that
\begin{eqnarray*}
Vf(u(x)) = a(x)|f(x)|^{p(x)} sgn(f(x)),\\
Wf(u(x)) = b(x)|f(x)|^{p(x)} sgn(f(x)), \\
Uf(u(x)) = c(x)|f(x)|^{p(x)} sgn(f(x)),\\
\end{eqnarray*}
with $c(x) = \frac{a(x)}{b(x)}$.
\end{thm}

\begin{proof}
Take $g \equiv 1$, and denote $Wg(x) = c(x)\in C(\R)$ and $Ug(x) = d(x)\in C(\R)$.
Then $V(f) = U(f) \cdot c(x) = d(x)\cdot W (f)$. From bijectivity of $V$ we see that $c$ and $d$ can never vanish, and that
\[ V(fg) = \frac{1}{c(x)d(x)}V(f)V(g). \]
Define $Tf = Vf/(cd)$ we have
\[ T(fg)= \frac{1}{cd}V(fg) = \frac{1}{cd}Vf\cdot \frac{1}{cd}Vg = Tf \cdot Tg. \]
Since $V$ is a bijection, so is $T$, and we may apply Theorem \ref{thm-main:Cmfld} to conclude that  there exists some continuous $p:M\to \RR_+$, and a
homeomorphism $u: M \to M$ such that
\begin{equation}
(Tf)(u(x)) = |f(x)|^{p(x)} sgn(f(x)).
\end{equation}
Therefore, letting $a(x) = c(u(x))d(u(x))$ and $b(x) = d(u(x))$, the proof is complete.
%\begin{eqnarray*}
%Vf(u(x)) = c(u(x))d(u(x))|f(x)|^{p(x)} sgn(f(x)),\\
%Wf(u(x)) = d(u(x))|f(x)|^{p(x)} sgn(f(x)), \\
%Uf(u(x)) = c(u(x))|f(x)|^{p(x)} sgn(f(x)).\\
%\end{eqnarray*}
\end{proof}

\begin{rem}
The version of the above theorem in which Fourier transform can be
applied (say, the $L_2(\R) \cap C(\R)\cap L_\infty(\R)$ case) has,
as usual, a direct consequence regarding the
 exchange of product and convolution. Assume $W,U,V$ satisfy
\[  V(f\cdot g) = W(f)*U(g).\]
Apply $\FF$ and get that
$V' = \FF V$, $W' = \FF W$ and $U' = \FF U$ satisfy
\[ V'(f \cdot g) = W' (f) \cdot U'(g), \]
which is equation \eqref{eq_multthreediff}. Of course, one needs some assumption on the range of $V$ to get that $V'$ is a bijetcion, and apply a modification of Theorem \ref{thm-gen}. The conclusion would be of the form
%\[ (\FF V f)(u(x)) = a(x)|f(x)|^{p(x)} sgn(f(x)),\]
%which means
\[ (V f)(u(x)) = \hat{a}(x)* \FF\left(|f(x)|^{p(x)} sgn(f(x))\right).\]
\end{rem}

\appendix
\section{Additive local operators on jet bundles}
\subsection{A review of jet bundles}

We briefly outline the basic definitions concerning jet bundles. For
more details, see \cite{Saunders}.

Let $M$ be a $C^n$ manifold ($0\leq n\leq \infty$), and fix a $C^n$
smooth real vector bundle $E$ over $M$. We denote by $\mathcal
O_k(E, U)$ the $C^k$ sections on $U\subset M$, and $\Gamma({\mathcal
O}^k(E))=\mathcal O_k(E, M)$ the global $C^k$ section of $E$. Also,
$\Gamma^k_c(E)$ will denote the compactly supported global $C^k$
sections of $E$. For our purposes, we really only need two cases:
$E=M\times\R$ and $E=M\times \C$. The $C^k$ sections are then simply
$C^k$ functions on $M$ with values in $\R$ or $\C$. Let $J^k=J^k(E)$
(with $k\leq n$) denote the associated $k$-jet bundle for which the
fiber over $x\in M$ is denoted $J^k_x$. We give two equivalent
definitions of jet bundles, and consider first the case $k<\infty$.
\\\\1. Consider the sheaf of modules $\mathcal O_k(E)$ of $C^k$
sections of $E$ over the sheaf of rings $\mathcal O_k(M):=\mathcal O_k(M\times \R)$. The stalk $\mathcal O_{k,x}$
at $x$ of $\mathcal O_k(M)$ is a local ring, with the maximal ideal $n_x=\{f_x(x)=0\}$.
We then define the space $J^k_x(E)$ of $k$-jets
at $x$ as the quotient $\mathcal O_{k,x}(E)/n_x^{k+1}\mathcal O_{k,x}(E)$. We denote the projection
\[J^k(x): \mathcal O_{k, x}(E)\to J^k_x(E)\] and for $k<n$ one also has the natural projections
\[P(k)(x):J^{k+1}_x\to J^k_x\] We then topologize the disjoint union $J^k=\bigcup_{x\in M}J^k_x$
by requiring all set-theoretic sections of the form $\sum_{j=1}^N a_j(x)J^kf_j(x)$ to be continuous, for $a_j\in C(M)$
and $f_j\in \Gamma(\mathcal O_k(E))$.
\\Informally, $J^kf(x)$ is the $k$-th Taylor polynomial of $f$, in a coordinate-free notation.
For instance, $J^0(E)=E$, and $J^1(M\times\R)=(M\times\R)\oplus T^*M$.
\\\\2. It is well known that if $x\in M$ is a critical point for a function $f\in C^2(M)$, i.e. $d_xf=0$,
then the Hessian of $f$ is well defined.
For a general bundle $E$, the same happens already in the first order: the 1-jet of a germ $s_x\in \mathcal O_{k,x}(E)$
is well defined at $x\in M$ (it is an element of $T_x^*M\otimes E_x$), given that $s(x)=0$. Proceeding by induction, one can show that
the notion $J^k_xs_x=0$ is well defined for all finite $k$, as well as for $k=\infty$.
The space of $k$-jets at $x$ is then defined as the quotient $\mathcal O_{k,x}(E)/\{s_x: J^{k+1}_xs_x=0\}$.
\\\\For $k=\infty$ (assuming $M$ is $C^\infty$) only the second definition applies.
Alternatively, one could generalize definition 1 as follows:
define the space $J_x^\infty (E)$ of $\infty$-jets at $x$ as the $n_x$-adic completion
of $\mathcal O_{k,x}(E)$. This is just the inverse limit of $J^k_x$ through the projections $P(k)$. Thus,
\[J_x^\infty (E)=\{(v_j)_{j=0}^\infty: P(j)(v_{j+1})=v_j\}\]
The equivalence of the two definition is known as Borel's lemma.
The set-theoretic bundle of jets $J=J^\infty$ will be the disjoint union of all fibers
$J_x^\infty (E)$. We denote $Jf=J^\infty f$ the section of $J$ defined by a $C^\infty$ section $f$ of $E$.
For $0\leq j<k\leq \infty$, denote $P(k,j):J^k\to J^j$ the
natural projection map, and $Q(j,k)=P(k,j)^*:(J^j)^*\to (J^k)^*$.
Also, denote $P_k=P(\infty, k)$, $Q_k=P_k^*$.  We
will assume that some inner product is chosen on $J^n$ for
$n<\infty$.
\\\\In general, there is no canonic map $J^k_x\to J^n_x$ when $k<n$. However,
when $k=0$ and $E=M\times V$ for some vector space $V$,
such a map does exist: We have $J^0_x\simeq E_0x\simeq V$ and all isomorphisms are canonical,
so one can choose a constant function with the given value and consider its $n$-jet. Such jets are called constant jets.
\\\\We list a few more basic properties of jets, which will not be used in the paper.
Assume that $E$ is trivial over $M$, with fiber either $\R$ or $\C$. The fiber $J^n_x$ with
$0\leq n\leq \infty$ is then naturally a local ring, with maximal ideal $
m_x=\{J^nf: f(x)=0\}= Ker P(n,0)$, $m_x^k=Ker P(n,k)$. Thus we get
an $m_x$-filtration of $J^n_x$. The induced
$m_x$-adic topology on $J^n_x$ is Hausdorff by definition, and
complete.
\subsection{$S^1$-bundles}
We would like to discuss separately the case $E=M\times S^1$, where
$S^1$ is considered as a Lie group. We make the following
definition: For $f=e^{i\theta} \in C^k(M, S^1)$, $J^kf:=J^k\theta$.
Note that $J^0f$ is only defined up to $2\pi$. Further, note that a
functional $c:J^k_x(M\times S^1)\to\R$ is induced from a functional
$\tilde c: J_x^k(M\times\R)\to\R$ which vanishes on constant jets.
Denote the bundle of such functionals by $(J^k)^*(M\times S^1,\R)$.
We also define a linear character $\chi:J^k_x(M\times S^1)\to S^1$
by $\chi=\exp(ic)$ where $c:J_x^k(M\times\R)\to\R$ is a linear
functional s.t. $c(v)=mv$ for constant jets $v$, for some fixed
$m\in\mathbb Z$. The bundle of linear characters $\chi$ (as well as
that of the corresponding functionals $c$) is denoted
$(J^k)^*(M\times S^1, S^1)$.

\subsection{Main Lemmas}
\begin{lem} Fix a $C^s$ manifold $M$, $0\leq s\leq\infty$, and a $C^s$ real vector bundle $E$ over $M$.
Assume $B:J^s(E) \to \R$  satisfies the following
conditions:
\\ (1) $B(x,u+v)=B(x,u)+B(x,v)$ for all $x\in M$, $u,v\in J^s_x$.
\\ (2) For all $f\in \Gamma^s_c(E)$, one has $B(J^s f)\in C(M)$.
\\ Then $B$ is linear in every fiber.
\end{lem}

\noindent{\bf Proof.}

Let $A=\{x\in M | B(x,\cdot): J_x^s\to\mathbb R\mbox{
is non-linear}\}$.
First we prove that $A$ has no accumulation points in $M$ .\\
Assume the contrary, i.e. $A\ni x_k\rightarrow x_\infty$. We can
assume that $x_k\neq x_l$ for $1\leq k<l\leq \infty$. By the
assumption, there exists a sequence $v_k\in J_{x_k}^s$ such
that the functions $B(x_k, tv_k)$ are additive and non-linear in
$t$.  We will construct in \ref{subsec:functions} a sequence of sections $f_k\in \Gamma^s(E)$ such that:
\\(a) $f_k$ is supported in a small neighborhood of $x_k$, such that
$\Supp(f_k)\cap \Supp(f_l)=\emptyset$ for $k\neq l$, and
$x_\infty\notin \Supp(f_k)$, for all $k<\infty$.
\\(b) $J^sf_k(x_k)=\epsilon_k v_k$ for some $0<\epsilon_k<1$.
\\(c) $|J^{\min(k,s)}f_k(x)|<2^{-k}$ for all $x\in M$ and $k<\infty$.
\\Given that, choose $0<t_k\to 0$ such that $|B(x_k, t_k \epsilon_k
v_k)|> 1$, using that non-linear additive functions are not locally bounded, and consider $f=\sum_{k=1}^\infty t_k f_k$. By condition
(c), $f\in \Gamma^s(E)$, and by (a), (b), $J^sf(x_k)=t_kJ^sf_k(x_k)=t_k\epsilon_k
v_k$ and $J^sf(x_\infty)=0$. Thus $B(x_k, t_kJ^sf(x_k))\to B(x_\infty,
J^sf(x_\infty)$ by condition (2) on $B$, but $|B(x_k, t_kJ^sf(x_k)|>1$
while $B(x_\infty, 0)=0$, a contradiction.
\\\\Next, we show that $A$ is empty. Indeed, take any sequence $x_k\notin
A$ converging to arbitrary $x_\infty\in M$. Fix any $v\in
J^s_{x_\infty}$, and choose $f\in \Gamma^s(E)$ with
$J^sf(x_\infty)=v$. Then for all $t\in\R$, $B(x_k, tJ^sf(x_k))\to
B(x_\infty, tJ^sf(x_\infty))$, i.e. $B(x_\infty, tv)$ is a pointwise
limit of continuous functions (of $t\in\R$), thus a measurable
function. Next use a theorem of Banach and Sierpinski, see \cite{Banach} and \cite{Sier} which states that if $B(x_\infty, tv)$ is a measurable function of $t$ and $B(x_\infty, \cdot)$
is additive, then it must be linear. $\hfill \square$

\begin{lem}\label{lemma6.2} Under the conditions of the lemma above
\\ (1) If $s<\infty$ then $B$ is a continuous section of $(J^s)^*$.
\\ (2) If $s=\infty$ then $B$ is locally finite dimensional and continuous in the following sense:
Denoting $F_n=\{x: B(x,\cdot)\in Image(Q_n)_x\}$, one has $F_n\subset F_{n+1}$ are
closed sets, and for all $K\subset M$ compact, there exists an $n$ such that $K\subset F_n$  and $B=Q_n (c_n)$ where $c_n$ is a continuous section of $(J^n)^*$ over $K$.
\end{lem}

\noindent{\bf Proof.}
(1) Fix some coordinate chart $U\subset M$, with a
trivialization of $E$ over $M$. Take $c_s(x)\in (J^s_x)^*$ s.t. $B(x,
v)=c_s(x)(P_s(v))=\sum_{|\alpha|\leq s} a_{\alpha, \beta}(x)
v_\beta^{(\alpha)}$ (here $\alpha$ parametrizes the order of the
derivative, and $\beta$ the coordinate in $E_x$) for $x\in U$, $v\in\ J_x^s$. Take $x_k\to x_\infty$ within $U$, with $x_\infty$ at
the origin. We claim that $a_{\alpha,\beta}(x_k)\to a_{\alpha,\beta}(x_\infty)$: this is straightforward by condition (2) if one
considers polynomial sections. Thus $c_s$ is continuous.\\
\noindent (2) Step 1. Let $C=\{x\in M | B(x,\cdot)\notin Image(Q_n) \mbox{ for
all } n<\infty\}$. We prove that $C$ has no accumulation points.
Like before, assume the contrary, i.e. $C\ni x_k\rightarrow
x_\infty$, and all $x_k$ are different for $k\leq \infty$. Choose a
sequence $v_k\in J_{x_k}^\infty$ with $P_k(v_k)=0$ and $B(x_k,
v_k)=1$. Again, we construct in \ref{subsec:functions} a sequence of functions
$g_k\in\Gamma^\infty(E)$ such that:
\\(a) $g_k$ is supported in a small neighborhood of $x_k$, such that
$\Supp(g_k)\cap \Supp(g_l)=\emptyset$ for $k\neq l$, and
$x_\infty\notin \Supp(g_k)$, for all $k$.
\\(b) $Jg_k(x_k)=v_k$.
\\(c) $|J^kg_k(x)|<2^{-(k-1)}$ for all $x\in M$.
\\Now consider $g=\sum_{k=1}^\infty g_k$. By condition (c), $g\in
\Gamma^\infty(E)$, $Jg(x_k)=Jg_k(x_k)=v_k$ and $Jg(x_\infty)=0$.
Thus $1=B(x_k, Jg(x_k))\to B(x_\infty, Jg(x_\infty))=0$ by condition
(2) on $B$, a contradiction.

\noindent
Step 2. Fix some coordinate chart $U\subset M$, with a
trivialization of $E$ over $M$. Then for $x\in F_n$ one has
$c_n(x)\in (J^n_x)^*$ s.t. $B(x,
v)=c_n(x)(P_n(v))=\sum_{|\alpha|\leq n} a_{\alpha, \beta}(x)
v_\beta^{(\alpha)}$ (here $\alpha$ parametrizes the order of the
derivative, and $\beta$ the coordinate in $E_x$) for $x\in F_n\cap
U$, $v\in\ J_x^\infty$. Take $x_k\to x_\infty$ within $U$, with
$x_k\in F_n$. We will show that $a_{\alpha,\beta}(x_k)$ converges as
$k\to\infty$: this is straightforward by condition (2) if one
considers polynomial sections. And so $c_n$ extends to a continuous
section of $(J^n)^*$ on the closure of $F_n$. This implies that for
all $v\in J^\infty_x$ one can choose any $f$ with $Jf(x_\infty)=v$,
and then by condition (2)
\[B(x_\infty, v)=B(x_\infty, Jf(x_\infty))=
\lim_{k\to\infty}B(x_k, Jf(x_k)) =\lim_{k\to\infty}c_n(x_k)
(J^nf(x_k))=\]
 \[=c_n(x_\infty)(J^nf(x_\infty))
=c_n(x_\infty) (v)\] i.e. $x_\infty\in F_n$. Thus, $F_n$ is closed.

\noindent
Step 3. We now show that $C$ is in fact empty. Assume otherwise, and
take $M\setminus C\ni x_k\to x_\infty$, $x_\infty\in C$. Since $F_k$
are all closed sets, we may assume $x_k\in M\setminus F_k$. We may
further assume that all $x_k$ are distinct. Now we simply repeat the
construction of step 1: Choose a sequence $v_k\in J_{x_k}^\infty$
with $P_k(v_k)=0$ and $B(x_k, v_k)=1$. Choose the functions $g_k$,
$g$ as before. Thus $Jg(x_k)=v_k$ and $Jg(x_\infty)=0$, so $1=B(x_k,
Jg(x_k))\to B(x_\infty, Jg(x_\infty))=0$ by condition (2) on $B$, a
contradiction.

\noindent
Step 4. Finally, if $K\subset M$ is compact, and $F_k\cap K$ is a
strictly increasing sequence of subsets, one can choose a converging
sequence $K\setminus F_k\ni x_k\to x_\infty$, which is impossible by
step 3. Thus $K\subset F_n$ for large $n$.

$\hfill \square$

\subsection{The construction of the function families.}\label{subsec:functions}
We finally construct the functions $f_k$, $g_k$ which we used in the lemmas above, for some fixed
$k$. We do this by mimicking the proof of Borel's lemma. Since the construction is local, we can assume $E$ is trivial.
For simplicity, we further assume $M=\R$, and $E=M\times\R$. We will denote
$v=v_k$. For convenience, assume $x_k=0$. Take
\[\delta=\min\left(\frac{1}{2}\min\{|x_m-x_k| : 1\leq m\leq
\infty, m\neq k.\}\right)\] Fix a smooth, non-negative function
$h\in C^\infty(\R)$ such that $h(x)=0$ for $|x|>\delta$, and
$h(x)=1$ for $|x|\leq \delta/2$. Let $\psi_n(x)=x^nh(x)$, then
$\psi_n^{(j)}(0)=n!\delta_n^j$.
\\Define \[\lambda_n=\max\{1, |v^{(j)}|, \sup|\psi_n|, \sup |\psi'_n|,...,\sup
|\psi^{(n)}_n|\}\]
\[\phi_n=\frac{v^{(n)}}{n!\lambda_n^n}\psi_n(\lambda_nx)\]
Then
\[\phi_n^{(j)}=\frac{v^{(n)}}{n!\lambda_n^{n-j}}\psi_n^{(j)}(\lambda_nx)\] and $\phi_n^{(j)}(0)=v^{(j)}\delta_n^j$.
Take $H=\sum_{n=0}^\infty \phi_n(x)$. This is a $C^\infty$ function,
since
\[\sum |\phi_n^{(j)}|\leq \sum_{n=0}^{j+1}\frac{|v^{(n)}|}{n!\lambda_n^{n-j}}\sup|\psi_n^{(j)}|
+\sum_{n=j+2}^\infty\frac{1}{n!}\frac{1}{\lambda_n^{n-j-2}}\frac{|v^{(n)}|}{\lambda_n}\frac{\sup|\psi_n^{(j)}|}{\lambda_n}\]
and $JH(0)=v$. \\Now we consider our two cases separately: For an
arbitrary $v_k\in\R^\infty$, take
\[\epsilon_k=\frac{1}{2^k(1+\max\{\sup|H|, \sup|H'|,...,\sup|H^{(k)}|\})}\] and $f_k=\epsilon_k H$ satisfies (a)-(c).
\\For $v_k\in Ker(P_k)$, we have for all $j<k$
\[|H^{(j)}|\leq \sum_{n=k+1}^\infty\frac{1}{n!}\leq \frac{1}{2^{k-1}}\]
 so $g_k=H$ satisfies (a)-(c) $\hfill \square$

\subsection {Replacing $\R$ with $S^1$}
Next we state without proof several variants of the lemmas above. To see how the proofs above adapt to those $S^1$ cases, see \cite{AAFM}.

\begin{lem}\label{lemRS} Fix a $C^\infty$ manifold $M$, and $E=M\times\R$.
Assume $B:J^\infty \to S^1$ (here $S^1$ is a Lie group) satisfies the following
conditions:
\\ (1) $B(x,u+v)=B(x,u)B(x,v)$ for all $x\in M$, $u,v\in J^\infty_x$.
\\ (2) For all $f\in C_c^\infty(M)$, one has $B(J^\infty f)\in C(M)$.
\\ Then $B(x,u)=\exp(i\langle c(x), u\rangle)$ where $c\in J^*$. Moreover, $M=\bigcup_{n=0}^\infty F_n$ where $F_n=\{x:
c(x)\in Image(Q_n)_x\}$. The sets $F_n\subset F_{n+1}$ are
closed, and for all $K\subset M$ compact, one has $K\subset F_n$ for
some $n$, and $B$ is a continuous section of $(J^n)^*$ over $K$.
\end{lem}

\begin{lem}\label{lemSR} Fix a $C^\infty$ manifold $M$, and $E=M\times S^1$.
Assume $B:J^\infty (E) \to \R$ satisfies the following
conditions:
\\ (1) $B(x,uv)=B(x,u)+B(x,v)$ for all $x\in M$, $u,v\in J^\infty_x$.
\\ (2) For all $f\in C_c^\infty(M, S^1)$, one has $B(J^\infty f)\in C(M)$.
\\ Then $B(x,v)=\langle c(x), v\rangle)$ where $c\in J^*$ vanishes on constant jets.
Moreover, $M=\bigcup_{n=0}^\infty F_n$ where $F_n=\{x:
c(x)\in Image(Q_n)_x\}$. The sets $F_n\subset F_{n+1}$ are
closed, and for all $K\subset M$ compact, one has $K\subset F_n$ for
some $n$, and $B$ is a continuous section of $(J^n)^*$ over $K$.
\end{lem}

\begin{lem}\label{lemSS} Fix a $C^\infty$ manifold $M$, and $E=M\times S^1$.
Assume $B:J^\infty (E) \to S^1$ satisfies the following
conditions:
\\ (1) $B(x,uv)=B(x,u)B(x,v)$ for all $x\in M$, $u,v\in J^\infty_x$.
\\ (2) For all $f\in C_c^\infty(M, S^1)$, one has $B(J^\infty f)\in C(M, S^1)$.
\\ Then $B(x,v)=\exp(i\langle c(x), v\rangle)$ where $c\in J^*$. Moreover, $M=\bigcup_{n=0}^\infty F_n$ where $F_n=\{x:
c(x)\in Image(Q_n)_x\}$. The sets $F_n\subset F_{n+1}$ are
closed, and for all $K\subset M$ compact, one has $K\subset F_n$ for
some $n$, and $B$ is a continuous section of $(J^n)^*$ over $K$. Also, $c(x)(v)= mv$ for constant jets $v$, with  $m\in\mathbb Z$.
\end{lem}

\begin{rem} One also has the $C^s$ simpler versions of those lemmas, with
$s<\infty$.\end{rem}

\subsection{A consequence of Peetre's Theorem}
In the following, $M$ is a real smooth manifold, and $E$, $F$ are
smooth vector bundles over $M$. First, recall Peetre's Theorem
\cite{peetre}:

\begin{thm*}\label{thm:peetre}
Let $Q:\Gamma^\infty(E)\to\Gamma^\infty(F)$ be a linear operator
which decreases support: $\supp(Qs)\subset \supp(s)$. Then for every
$x\in M$ there is an integer $k\geq 0$ and an open neighborhood
$U\ni x$ s.t. $P$ restricts to a differential operator of order $k$
on $\Gamma^\infty(U, E)$.
\end{thm*}

\begin{lem}\label{peetre}
Let $P:\Gamma^\infty(E)\to\Gamma^\infty(E)$ be an invertible
differential operator, which for all $x\in M$ induces an isomorphism
on the stalk $\mathcal O_x(E)$. Then $P$ is of order $0$.
\end{lem}

\noindent{\bf Proof.} Denote $Q=P^{-1}$. We claim that $Q$ is a
local operator (i.e. $Qf(x)$ only depends on the germ $f_x$ of
$f\in\Gamma^\infty(E)$). Indeed, assume $f_x=0$ and take $g=Qf$.
Then $(Pg)_x=f_x=0$, implying $g_x=0$, as required. In particular,
$Q$ does not increase supports: $\supp(Qf)\subset \supp(f)$. By
Peetre's theorem, $Q$ is locally a differential operator. Thus in
small neighborhoods $U\subset M$, $P$ and $Q$ are two differential
operators that are inverse to each other, and we conclude they have
order $0$. $\hfill \square$

\section{Bijections of open sets preserving intersection}

In this appendix we give another lemma in the same spirit, which may
be of use in other settings. Denote by $U(\R)$ the open subsets of
$\R$.

\begin{lem}
Assume $F: U(\R) \to U(\R)$ is a bijection which satisfies
\[ F(U_1 \cap U_2) = F(U_1) \cap F(U_2). \]
Then $F$ is given by an open point map $u: \R\to \R$, \[ F(U) = \{
u(x): x \in U\}. \] \end{lem}

\noindent{\bf Proof.} First we claim that $F$ is an order
isomorphism. Indeed, (using injectivity)
\[ U_1 \subset U_2
\Leftrightarrow U_1 \cap U_2 = U_1 \Leftrightarrow F(U_1) \cap
F(U_2) = F(U_1)\Leftrightarrow F(U_1) \subset F(U_2). \] Next, note
that this implies that \[ F(U_1 \cup U_2) = F(U_1) \cup F(U_2). \]
Indeed, denote (using surjectivity) $F(U_3)  = F(U_1) \cup F(U_2)$,
then $F(U_1) \subset F(U_1 \cup U_2)$ and $F(U_2) \subset F(U_1 \cup
U_2)$ and thus $F(U_3) \subset F(U_1 \cup U_2)$ and so $U_3\subset
U_1 \cup U_2$. On the other hand $U_1 \subset U_3$ and $U_2 \subset
U_3$, so that $U_1 \cup U_2 \subset U_3$. We get equality, thus $F$
preserves union too.

Next we claim that the set of special open sets $\{ A_x : x\in \R\}$
given by $A_x = \{ y \in \R : y\neq x\}$ is invariant under $F$.
Indeed, start by noticing that $F(\R) = \R$, as it satisfies that
$\R \cap A = A$ for every $A$. (Also that $\R \cup B  = \R$ for
every $B$) and these properties are preserved under $F$ and unique
to this subset.

Next, notice that a set $A_x$ satisfies $A_x \cup B = A_x$ or $\R$
for every $B$. Whereas every other open set $A$ has at least two
points which are not elements in it, and thus one may construct at
least three different sets as combinations $A \cup B$ for various
open $B$'s. Therefore, $F(A_x) = A_{u(x)}$. Notice that $F^{-1}$
satisfies the same conditions as $F$ and thus $u$ is a bijection.

Finally, for a set $A$ we know that for any $x \not\in A$
\[ F(A) = F(A \cap A_x) = F(A) \cap A_{u(x)}\]
and thus $u(x) \not\in F(A)$, so that $F(A) \subset u(A)$. However,
using the same argument for $u^{-1}$ we get that $F^{-1}(A) \subset
u^{-1}(A)$, and applying this to the set $F(A)$ we get $A \subset
u^{-1}(F(A))$ which means $u(A) \subset F(A)$ and the proof is
complete. $\hfill \square$

\begin{rem} Clearly $u$ is a continuous bijection.
It is also clear that the setting of $\R$ can be vastly generalized.  \end{rem}

\end{document}